\documentclass[10pt]{article}

\usepackage{graphicx}
\usepackage{mathptmx}
\usepackage{amsmath,amsthm}
\usepackage{amssymb}

\newcommand{\abs}[1]{\left|#1\right|}
\newcommand{\norm}[1]{\left|\left|#1\right|\right|}
\newcommand{\normlkp}[3]{\norm{#1}_{L^{#2,#3}}}

\newtheorem{theorem}{Theorem}[section]
\newtheorem{corollary}[theorem]{Corollary}
\newtheorem{lemma}[theorem]{Lemma}
\newtheorem{proposition}[theorem]{Proposition}
\newtheorem{conjecture}[theorem]{Conjecture}

\newtheorem*{prty1}{Property 1}{\bf}{\it}
\newtheorem*{prty2}{Property 2}{\bf}{\it}
\newtheorem*{prty1_1}{Property 1'}{\bf}{\it}

\begin{document}

\title{$L^p$ Error Estimates for Approximation by Sobolev Splines and Wendland Functions on $\mathbb{R}^d$\thanks{This paper contains work from the author's dissertation written under the supervision of Professors F. J. Narcowich and J. D. Ward at Texas A\&M University.  This research was supported by grant DMS-0807033 from the National Science Foundation.}}

\author{John Paul Ward\thanks{Biomedical Imaging Group, EPFL, CH-1015 Lausanne, Switzerland } }

\maketitle

\begin{abstract}
It is known that a Green's function-type condition may be used to derive rates for approximation by radial basis functions (RBFs).  In this paper, we introduce a method for obtaining rates for approximation by functions which can be convolved with a finite Borel measure to form a Green's function.  Following a description of the method, rates will be found for two classes of RBFs.  Specifically, rates will be found for the Sobolev splines, which are Green's functions, and the perturbation technique will then be employed to determine rates for approximation by Wendland functions.
\end{abstract}

\section{Introduction}
Radial basis function (RBF) approximation is primarily used for constructing approximants to functions that are only known at discrete sets of points.  
Some advantages of this theory are that RBF approximation methods theoretically work well in arbitrarily high dimensional spaces, where other methods break down, and ease of implementation. Even though the theory tells us that there is no dimensional bound on the applicability of RBF techniques, in practice, interpolating large sets of data coming from high dimensional spaces can be computationally expensive.  Fortunately, efficient algorithms \cite{r:pow:1,r:wright,r:barb:1} have been developed to compensate for such difficulties and allow for wider application of RBF methods.  As a result of their favorable properties and efficient implementation schemes, RBF techniques are being used to solve a variety of applied problems, some examples being problems in data mining \cite{r:dm:1,r:dm:2}, statistical learning theory \cite{r:hof}, and numerical partial differential equations. In particular, Flyer and Wright have shown that RBFs perform well in problems from mathematical geosciences, \cite{fly1,fly2}.

The initial approach to scattered data approximation on $\mathbb{R}^d$ concerned the stationary setting, with RBFs being scaled to be proportional to the fill distance.  In this paper, we will study the alternative non-stationary setting, which is relatively new and still contains many open problems. For such an approximation, one approach to proving error estimates is to exploit a Green's function representation.  For example, this technique was used by Hangelbroek for analyzing approximation by thin-plate splines on the unit disk, \cite{r:hang:1,r:hang:2}. More recently, DeVore and Ron derived rates for the approximation of $L^p$ functions by RBF spaces $S_X(\Phi)$, \cite{r:dev:1}.  Here, $S_X(\Phi)$ denotes a collection of linear combinations of translates of an RBF $\Phi$, where the translations come from a scattered discrete set of points $X \subset \mathbb{R}^d$.  In order to derive these rates, the functions $\Phi$ were required to satisfy a Green's function-type condition.  The focus of this paper will be to extend a result of DeVore and Ron to include certain non-Green's functions.  Specifically, the Wendland functions, which have some nice properties, such as compact support, and are hence popular for applications \cite{r:diag,r:gelas}, do not fit the framework of \cite{r:dev:1}.  Therefore, this paper will address a generalization of a result from \cite{r:dev:1} that will allow us to derive approximation rates for the Wendland functions.  Additionally, rates for approximation by Sobolev splines will be determined.

To begin we will prove approximation rates for functions that can be convolved with a finite Borel measure to form a Green's function.  The motivation for this approach is the success of Mhaskar, Narcowich, Prestin and Ward in deriving rates for perturbations of Green's functions on the sphere, \cite{r:bernSn}.  In fact, they were able to find rates for the Wendland functions restricted to the sphere. After establishing this general result, we will provide some examples.  To be precise, polynomial reproducing functionals will be used to show that the Sobolev splines and Wendland functions can be well approximated by local translates of themselves, and an error bound for approximation by some particular Sobolev splines and Wendland functions will be derived.

\subsection{Notation and Definitions}

The Fourier transform and Laplace transform are fundamental tools for proving many results in approximation theory.  We will use them to characterize the RBFs under consideration and to prove several results.  Throughout this paper, we will use the following conventions.  Given a function $f:\mathbb{R}^d\rightarrow \mathbb{R}$ in $L^1$, its Fourier transform $\hat{f}$ will be defined by
\begin{equation*} 
 \hat{f}(\omega)=\frac{1}{(2\pi)^{d/2}}\int_{\mathbb{R}^d} f(x)e^{-ix\cdot\omega} dx.
\end{equation*}
Given a function $f:[0,\infty)\rightarrow \mathbb{R}$ that grows no faster than $e^{at}$ for some $a>0$, its Laplace transform will be defined by
\begin{equation*} 
 \mathcal{L}f(z)=\int_0^\infty f(t)e^{-zt} dt.
\end{equation*}

% function spaces 
Let $E$ be a subset of $\mathbb{R}^d$, then we denote by $C^k(E)$ the collection of real valued functions defined on $E$ that have continuous partial derivatives up to order $k$.  The set of functions in $C^k(\mathbb{R^d})$ that converge to 0 at infinity will be denoted by $C_0^k$,  and we represent the compactly supported elements of $C^k(\mathbb{R}^d)$ by $C_c^k$.

The Schwartz class $\mathcal{S}$ of functions on $\mathbb{R}^d$ is defined as follows.  A function $f:\mathbb{R}^d\rightarrow \mathbb{R}$ is said to be of Schwartz class if for all multi-indices $\alpha$ and $\beta$, there exists a constant $C_{\alpha,\beta}>0$ such that 
\begin{equation*}
 \abs{x^\alpha D^\beta f(x)} \leq C_{\alpha,\beta}
\end{equation*}
for all $x \in \mathbb{R}^d$.

We shall use the standard definition for the $L^p$ spaces.  A Lebesgue measurable function $f:\mathbb{R}^d\rightarrow \mathbb{R}$ is in $L^p(\mathbb{R}^d)$  for $1\leq p <\infty$ if
\begin{equation*}
 \norm{f}_p=\left(\int_{\mathbb{R}^d} \abs{f(x)}^pdx \right)^{1/p} <\infty,
\end{equation*}
and a Lebesgue measurable function $f$ is said to be in $L^\infty(\mathbb{R}^d)$ if 
\begin{equation*}
 \norm{f}_\infty=\text{ess} \sup_{x\in \mathbb{R}^d} \abs{f(x)} <\infty.
\end{equation*}
Additionally a measurable function $f$ is said to be locally integrable, denoted by $L_\text{loc}^1(\mathbb{R}^d)$, if 
\begin{equation*}
 \int_{E} \abs{f(x)}dx <\infty
\end{equation*}
for each bounded, measurable set $E \subseteq \mathbb{R}^d$.

% measures
%-------------------------------------------------------------------------------------------
The space of finite Borel measures on $\mathbb{R}^d$ will be denoted by $M(\mathbb{R}^d)$.  This space is equipped with a norm defined by
\begin{equation*}
 \norm{\mu}=\abs{\mu}(\mathbb{R}^d),
\end{equation*}
 where $\abs{\mu}$ is the total variation of $\mu$.  
Given a measure $\mu \in M(\mathbb{R}^d)$ and an $L^p$ function $f$, their convolution is an $L^p$ function, and it satisfies the following generalization of Young's inequality, \cite[Proposition 8.49]{r:foll:1}:
\begin{equation}\label{eq:young}
\norm{f*\mu}_p\leq \norm{f}_p\norm{\mu}.
\end{equation}

 In deriving error estimates for the Wendland functions, we will need to consider a decomposition of measures.  Specifically, any measure $\mu \in M(\mathbb{R}^d)$ can be written as $\mu=\mu_a+\mu_s+\mu_d$, where $\mu_a$ is absolutely continuous with respect to Lebesgue measure, $\mu_d$ is a countable linear combination of Dirac measures, and $\mu_s=\mu-\mu_a-\mu_d$ is the singular continuous part of $\mu$. 
%-------------------------------------------------------------------------------------------

Our primary focus will be approximating functions that lie in subspaces of $L^p$ spaces.  The spaces that we will be mainly interested in are the Bessel-potential spaces $L^{k,p}(\mathbb{R}^d)$, which coincide with the standard Sobolev spaces $W^{k,p}(\mathbb{R}^d)$ when $k$ is a positive integer and $1<p<\infty$, cf. \cite[Section 5.3]{r:stein}.  The Bessel potential spaces are defined by 
\begin{equation*}
 L^{k,p} = \{f:\hat{f}=(1+\norm{\cdot}_2^2)^{-k/2}\hat{g}, g  \in L^p(\mathbb{R}^d)\}
\end{equation*}
for $1 \leq p \leq \infty$, and they are equipped with the norm
\begin{equation*}
 \normlkp{f}{k}{p} = \norm{g}_p.
\end{equation*}

We will also be working with smoothness spaces associated with linear operators.  If $T:C_c^k(\mathbb{R}^d)\rightarrow C_c(\mathbb{R}^d)$ is linear,  then we define a semi-norm and norm on $C_c^\infty(\mathbb{R}^d)$ by 
\begin{eqnarray*}
 \abs{f}_{W(L^p(\mathbb{R}^d),T)} &=& \norm{Tf}_p \\
 \norm{f}_{W(L^p(\mathbb{R}^d),T)} &=& \norm{f}_p +\abs{f}_{W(L^p(\mathbb{R}^d),T)}.
\end{eqnarray*}
The completion of the space $C_c^\infty(\mathbb{R}^d)$, with respect to the above norm, will be denoted by $W(L^p(\mathbb{R}^d),T)$.  In particular, notice that for any positive integer $n$ and $1\leq p<\infty$, we have
\begin{equation*}
 W(L^p(\mathbb{R}^d),(1-\Delta)^n)=L^{2n,p}(\mathbb{R}^d)
\end{equation*}
where $\Delta$ denotes the $d-$dimensional Laplacian.

The approximants used in this paper will be finite linear combinations of translates of an RBF $\Phi$, and the translations will come from a countable set $X \subset \mathbb{R}^d$.  The error of this approximation, which is measured in a Sobolev-type norm, depends on both the function $\Phi$ and the set $X$.  Therefore, given an RBF $\Phi$ and a set $X$, we define the RBF approximation space $S_X(\Phi)$ by 
\begin{equation*}
 S_X(\Phi) = \left\{ \sum_{\xi \in Y} a_\xi \Phi(\cdot-\xi): Y \subset X, \#Y<\infty\right\}\cap L^1(\mathbb{R}^d).
\end{equation*}
Note that the functions $\Phi$ and sets $X$ will be required to satisfy certain conditions so that we may prove results about rates of approximation.  The error bounds will be stated in terms of the fill distance
\begin{equation*}
   h_X = \sup_{x\in\mathbb{R}^d} \inf_{\xi\in X}\norm{x-\xi}_2,
\end{equation*}
which measures how far a point in $\mathbb{R}^d$ can be from $X$, and we additionally require the separation radius 
\begin{equation*}
  q_X = \frac{1}{2}\inf_{\genfrac{}{}{0pt}{}{\xi,\xi'\in X}{\xi \neq \xi'}}\norm{\xi-\xi'}_2
\end{equation*}
to be bounded in order to prevent accumulation points in $X$. We will therefore be working with sets for which the mesh ratio $\rho_X:= h_X/q_X$ is bounded by a constant;  such sets are called quasi-uniform.

\subsection{Radial Basis Functions}
% Sobolev splines

The two classes of RBFs that we will be interested in are the Sobolev splines and the Wendland functions.  The Sobolev Splines $G_\gamma:\mathbb{R}^d\rightarrow\mathbb{R}$ form a class of radial basis functions most easily defined in terms of their Fourier transforms:
\begin{equation*}
 \hat{G}_\gamma(\omega):=(1+\norm{\omega}_2^2)^{-\gamma/2}
\end{equation*}
for $\gamma>0$.  Notice that $G_\gamma$ is the Green's function for the (pseudo-)differential operator $(1-\Delta)^{\gamma/2}$.  These functions have been studied extensively in \cite{r:aron:1,r:aron:2}.  In these papers, it is shown that $G_\gamma$ is an analytic function except at 0, and several useful asymptotic approximations are determined.  Those necessary for our results are listed in the following proposition.

\begin{proposition}\label{pr:ss2}\textbf{\em (\cite[p.253]{r:aron:2})}
 For any multi-index $\alpha$, there are positive constants $C$ such that
\begin{equation*}
\abs{D^\alpha G_\gamma(x)}\leq
\begin{cases}
 C_{\gamma} \norm{x}_2^{\gamma-d-\abs{\alpha}} & \text{ for } \abs{\alpha} \geq \gamma - d \text{ and } \abs{\alpha} \text{ odd } \\
 C_{\gamma,\alpha,d} \norm{x}_2^{\gamma-d-\abs{\alpha}} & \text{ for } \abs{\alpha} > \gamma - d \text{ and } \abs{\alpha} \text{ even } \\
 C & \text{ for } \abs{\alpha} \leq \gamma - d \text{ and } \abs{\alpha} \text{ odd } \\
 C_{\gamma,\alpha,d}  & \text{ for } \abs{\alpha} < \gamma - d \text{ and } \abs{\alpha} \text{ even } 
\end{cases}
\end{equation*}
\end{proposition}

% Wendland fcns

Next, the Wendland functions compose a class of compactly supported RBFs that are radially defined as piecewise polynomials, and some examples are provided in Table \ref{tab:1}. Wendland's book \cite{r:wen:2} (particularly chapter 9) provides a detailed analysis of these functions and some of their approximation properties; however, an exact form for the Fourier transforms of these functions has not previously been found. Since we will need this information for our approximation analysis, we begin by computing the Fourier transforms of these functions. The Wendland functions $\Phi_{d,k}$ are determined by a dimension parameter $d$ and a smoothness parameter $k$, and they lie in $C^{2k}(\mathbb{R}^d)$. In Table \ref{tab:1} and in what follows, $r$ will be used to denote $\norm{x}_2$, and $\doteq$ will be used to indicate equality up to some positive constant factor. 

\begin{table} 
\caption{Examples of Wendland functions}
\label{tab:1}  
 \begin{tabular}{l c}
  \hline \hline
 Function & Smoothness \\ \hline
 $\Phi_{1,0}(x)=(1-r)_{+}$  & $C^0$ \\
 $\Phi_{1,1}(x)\doteq(1-r)_{+}^3(3r+1)$  & $C^2$ \\
 $\Phi_{1,2}(x)\doteq(1-r)_{+}^5(8r^2+5r+1)$  & $C^4$ \\ \hline
 $\Phi_{3,0}(x)=(1-r)_{+}^2$  & $C^0$ \\
 $\Phi_{3,1}(x)\doteq(1-r)_{+}^4(4r+1)$  & $C^2$ \\
 $\Phi_{3,2}(x)\doteq(1-r)_{+}^6(35r^2+18r+3)$  & $C^4$ \\
 $\Phi_{3,3}(x)\doteq(1-r)_{+}^8(32r^3+25r^2+8r+1)$  & $C^6$ \\ 
   \hline \hline
 \end{tabular}
\end{table}

We will now derive an explicit form of the Fourier transform of $\Phi_{d,k}$ in the case $d$ is odd.  Using the notation of \cite[Section 10.5]{r:wen:2}, let $d=2n+1$ and $m=n+k$.  Then by \cite[Lemma 6.19]{r:wen:2} and the definition of $\Phi_{d,k}$, we have
\begin{equation}\label{eq:wft}
 \hat{\Phi}_{d,k}(x)= B_mf_m(r)r^{-3m-2} 
\end{equation}
where $B_m$ is a positive constant and the Laplace transform of $f_m$ satisfies
\begin{equation*}
 \mathcal{L}f_m(r)=\frac{1}{r^{m+1}(1+r^2)^{m+1}}.
\end{equation*}
In order to find the inverse Laplace transform of the above expression, we will make use of partial fractions.  First, note that there exist constants $\alpha_j$, $\beta_j$, and $\gamma_j$ such that 
\begin{equation}\label{eq:lt_1}
 \frac{1}{s^{m+1}(1+s^2)^{m+1}} = \sum_{j=0}^{m} \frac{\alpha_j}{s^{j+1}} + \sum_{j=0}^{m} \frac{\beta_j}{(s+i)^{j+1}} + \sum_{j=0}^{m} \frac{\gamma_j}{(s-i)^{j+1}},
\end{equation}
and this decomposition is unique.  Now for any real $s$, the expression on the left is real.  Therefore, taking the complex conjugate of both sides, we get
\begin{equation*}
 \frac{1}{s^{m+1}(1+s^2)^{m+1}} = \sum_{j=0}^{m} \frac{\bar{\alpha}_j}{s^{j+1}} + \sum_{j=0}^{m} \frac{\bar{\beta}_j}{(s-i)^{j+1}} + \sum_{j=0}^{m} \frac{\bar{\gamma}_j}{(s+i)^{j+1}}.
\end{equation*}
Uniqueness of the decomposition then implies that for each $j$: 
\begin{center}
\begin{enumerate}
 \item [(i)]  $\alpha_j$ is real
 \item [(ii)] $\beta_j=\bar{\gamma}_j$.
\end{enumerate}
\end{center}

To further characterize the coefficients, we replace $s$ by $-s$ in (\ref{eq:lt_1}). First, we have
\begin{equation*}
 \frac{(-1)^{m+1}}{s^{m+1}(1+s^2)^{m+1}} = \sum_{j=0}^{m} \frac{(-1)^{j+1}\alpha_j}{s^{j+1}} + \sum_{j=0}^{m} \frac{(-1)^{j+1}\beta_j}{(s-i)^{j+1}} + \sum_{j=0}^{m} \frac{(-1)^{j+1}\bar{\beta}_j}{(s+i)^{j+1}},
\end{equation*}
and therefore
\begin{equation}
 \frac{1}{s^{m+1}(1+s^2)^{m+1}} = \sum_{j=0}^{m} \frac{(-1)^{j+m}\alpha_j}{s^{j+1}} + \sum_{j=0}^{m} \frac{(-1)^{j+m}\beta_j}{(s-i)^{j+1}} + \sum_{j=0}^{m} \frac{(-1)^{j+m}\bar{\beta}_j}{(s+i)^{j+1}}.
\end{equation}
Again using the uniqueness of the partial fraction decomposition, it follows that
\begin{enumerate}
 \item [(i)] $(-1)^{j+m}\alpha_j=\alpha_j$
 \item [(ii)] $(-1)^{j+m}\beta_j=\bar{\beta}_j$
\end{enumerate}
for each $j$.  The first property implies that $\alpha_j=0$ for either all odd $j$ or all even $j$.  The second property tell us that $\beta_j$ is real when $j+m$ is even, and it is imaginary when $j+m$ is odd. 

We now compute $f_m$ as the inverse Laplace transform of the sum in (\ref{eq:lt_1}).
\begin{equation*}
 f_m(r) = \sum_{j=0}^{m} \frac{\alpha_j}{j!} r^j+ \sum_{j=0}^{m}\frac{\beta_j}{j!} r^je^{-ir}  + \sum_{j=0}^{m}\frac{\bar{\beta}_j}{j!} r^je^{ir}
\end{equation*}
Now if $m$ is odd, we have $m=2l+1$ and
 \begin{eqnarray*}
 f_{m}(r) = \sum_{j=0}^{l} \frac{\alpha_{2j+1}}{(2j+1)!} r^{2j+1} &+& \sum_{j=0}^{l}\frac{\beta_{2j+1}}{(2j+1)!} r^{2j+1}(e^{-ir}+e^{ir}) \\
 &-& \sum_{j=0}^{l}\frac{\beta_{2j}}{(2j)!} r^{2j}(-e^{-ir}+e^{ir}), 
\end{eqnarray*}
which can be simplified to
 \begin{equation*}
 f_{2l+1}(r) = \sum_{j=0}^{l} \frac{\alpha_{2j+1}}{(2j+1)!} r^{2j+1}+ \sum_{j=0}^{l}\frac{2\beta_{2j+1}}{(2j+1)!} r^{2j+1}\cos(r)  - \sum_{j=0}^{l}\frac{2i\beta_{2j}}{(2j)!} r^{2j}\sin(r). 
\end{equation*}
Similarly, when $m=2l$ we have
 \begin{equation*}
 f_{2l}(r) = \sum_{j=0}^{l} \frac{\alpha_{2j}}{(2j)!} r^{2j}- \sum_{j=0}^{l-1}\frac{2i\beta_{2j+1}}{(2j+1)!} r^{2j+1}\sin(r)  + \sum_{j=0}^{l}\frac{2\beta_{2j}}{(2j)!} r^{2j}\cos(r). 
\end{equation*}

\begin{lemma}
 Let $\Phi_{d,k}$ be a Wendland function with $d$ odd, and define $n$ and $m$ by $d=2n+1$ and $m=n+k$. Then the exact form of the Fourier transform of $\Phi_{d,k}$ is found by substituting the above representations of $f_m$ into $\hat{\Phi}_{d,k}(x)= B_mf_m(r)r^{-3m-2}$.
\end{lemma}

We can now determine properties of $\hat{\Phi}_{d,k}$ that will be required for the error bounds.  The proof is omitted as the result follows easily from the formulas computed for $f_m$.
\begin{proposition}\label{pr:wend_ft_1}
 Let $d$ and $k$ be non-negative integers with $d$ odd, and consider $\hat{\Phi}_{d,k}(x)= B_mf_m(r)r^{-3m-2}$.  Then $(\cdot)^{-3m-2}f_m$ is an analytic function, and 
\begin{equation*}
  r^{-3m-2}f_m(r)= O(r^{-2m-2})
\end{equation*}
as $r\rightarrow \infty$.
\end{proposition}

% d=1
%---------------------------------------------------------------------------------------------

Let us now take a closer look at the 1-dimensional case.
\begin{proposition}\label{pr:wend_ft}
 If $k \in \mathbb{N}$, then there exists $C\in\mathbb{R}$ and $h\in L^1(\mathbb{R})$ such that
\begin{equation}\label{eq:wend_1d_eq}
 x^{2k+2}\hat{\Phi}_{1,k}(x)= B_k\left(\frac{1}{k!}+\frac{(-1)^{k+1}}{k!2^k}\cos(x)+C\frac{\sin(x)}{x}+\hat{h}(x)\right),
\end{equation}
where $\hat{h}$ is the Fourier transform of $h$.
\end{proposition}
\begin{proof}
 Since the case where $k$ is odd is similar to the case where $k$ is even, we will only prove the former. Recall that $\hat{\Phi}_{1,k}(x)= B_kf_k(r)r^{-3k-2}$, so let us begin by examining the function $f_k$.  For $k=2l+1$ and some constants $a_j$, $b_j$, and $c_j$, we have 
\begin{equation}\label{eq:decay}
 r^{-1}f_k(r) = \sum_{j=0}^{l} a_j r^{2j}+ \sum_{j=0}^{l}b_j r^{2j}\cos(r)  - \sum_{j=0}^{l}c_j r^{2j}\frac{\sin(r)}{r}.
\end{equation}
We can then define an analytic function $\tilde{f}_k:\mathbb{R}\rightarrow \mathbb{R}$ by 
\begin{equation*}
 \tilde{f}_k(x) = \abs{x}^{-1}f_k(\abs{x}),
\end{equation*}
and we will have $x^{2k+2}\hat{\Phi}_{1,k}(x)= B_kx^{-k+1}\tilde{f}_k(x)$.  Now since $\Phi_{1,k} \in L^1(\mathbb{R})$ and $\hat{\Phi}_{1,k}(x)= B_kf_k(r)r^{-3k-2}$,  $\tilde{f}_k(x)$ must have a zero of order $3k+1$ at 0, and therefore $\tilde{f}_k(x)$ has a power series of the form
\begin{equation}\label{eq:ftil_1}
 \tilde{f}_k(x) = \sum_{j=3k+1}^\infty d_j x^j.
\end{equation}

In order to verify \eqref{eq:wend_1d_eq}, we first need to determine some of the coefficients in \eqref{eq:decay}.  From our previous work, we know that $c_l\in \mathbb{R}$, and we can find $a_l$ and $b_l$ as follows.  In the partial fraction decomposition (\ref{eq:lt_1}), multiply both sides by $s^{m+1}(1+s^2)^{m+1}$.  By substituting the values $s=0$ and $s=-i$, we find that $\alpha_m=1$ and $\beta_m=(-1)^m/2^{m+1}$, and therefore $a_l= 1/k!$ and $b_l=(-1)^k/(k!2^{k+1})$.

We can now finish the proof by showing that 
\begin{equation*}
  \tilde{h}_k(x):=x^{-k+1}\tilde{f}_k(x)- \left(\frac{1}{k!}+\frac{(-1)^{k+1}}{k!2^k}\cos(x)-c_l\frac{\sin(x)}{x}\right)
\end{equation*}
is the Fourier transform of an $L^1$ function. Since $\tilde{h}_k(x)$ is identically 0 for $k=1$, we need only consider $k\geq 3$.  This can be verified by determining that $\tilde{h}_k(x)$ has two continuous derivatives in $L^1$, cf. \cite[p. 219]{r:foll:2}.  Considering the representation \eqref{eq:ftil_1}, it is clear that $\tilde{h}_k(x)$ has two continuous derivatives, and the decay of these functions can be bounded using \eqref{eq:decay}.
\qed
\end{proof}

%------------------------------------------------------------------------------------------------

\section{General Error Estimates}

In the paper \cite{r:dev:1}, the authors made use of a Green's function-type condition in order to determine rates of approximation, cf. \cite[Property A2]{r:dev:1}.  They were concerned with approximating functions in $L^p(\mathbb{R}^d)$ by approximation spaces $S_X(\Phi)$, where $X$ has no accumulation points and $h_X$ is finite.  The essence of their argument is as follows. 

\begin{prty1}
 Suppose $T:C_c^k(\mathbb{R}^d)\rightarrow C_c(\mathbb{R}^d)$ is a linear operator, $\Phi\in L_\text{loc}^1(\mathbb{R}^d)$, and for any $f \in C_c^k(\mathbb{R}^d)$ we have 
\begin{equation}\label{eq:grpr1}
 f = \int_{\mathbb{R}^d}Tf(t)\Phi(\cdot-t)dt.
\end{equation}
\end{prty1}
Now, given an $f \in C_c^k(\mathbb{R}^d)$, we can form an approximant by replacing $\Phi$ with a suitable kernel in (\ref{eq:grpr1}).  Since we are interested in approximating $f$ by $S_X(\Phi)$, the kernel should have the form
\begin{equation}\label{eq:grk}
 K(\cdot,t) = \sum_{\xi \in X(t)} A(t,\xi) \Phi(\cdot-\xi).
\end{equation}
The collection of possible kernels is restricted by requiring 
\begin{equation*}
 X(t) \subset B(t,Ch_X)\cap X
\end{equation*}
 for a constant $C>0$ and additionally requiring $A(t,\xi)$ to be in $L^1$ for all $\xi$.  An essential ingredient for deriving approximation rates is showing that $\Phi$ can be well approximated by $K$.  We therefore define the error kernel
\begin{equation}
 E(x,t):=\Phi(x-t)-K(x,t).
\end{equation}
and assume the following property, cf. \cite[Property A4]{r:dev:1}.

\begin{prty2}
 Given $\Phi$, there exists a kernel $K$ of the form (\ref{eq:grk}) and constants $l>d$, $\kappa>0$, and $C>0$ such that
\begin{equation*}
 \abs{E(x,t)} \leq C h_X^{\kappa-d} \left(1+\frac{\norm{x-t}_2}{h_X}\right)^{-l}
\end{equation*}
for all $x,t\in \mathbb{R}^d$.
\end{prty2}

While this approach can be used to provide estimates for some popular RBFs, e.g. the thin-plate splines, not all RBFs are Green's functions.  Therefore our goal will be to extend the class of applicable RBFs.  Note that in \cite{r:dev:1}, the authors were concerned with proving error bounds that account for the local density of the data sites.  As our goal is to extend the class of applicable RBFs, we will not discuss this more technical approach; instead we will assume $X$ is quasi-uniform, and our bound will be written in terms of the global density parameter $h_X$.

We will now show that it is possible to replace Property 1 by a condition that only requires $\Phi$ to be ``close'' to a function that satisfies this property.  Specifically, we will consider RBFs $\Phi$ that satisfy the following.

\begin{prty1_1}
 Let $T:C_c^k(\mathbb{R}^d)\rightarrow C_c(\mathbb{R}^d)$ and $G\in L_\text{loc}^1(\mathbb{R}^d)$ be a pair satisfying Property 1, and suppose $G= \Phi*\mu_n +\nu_n$ where $\Phi\in L_\text{loc}^1(\mathbb{R}^d)$, $\mu_n$ is a sequence of compactly supported, finite Borel measures with $\norm{\mu_n}$ bounded by a constant, and $\nu_n$ is a sequence of finite Borel measures with $\norm{\nu_n}$ converging to 0.
\end{prty1_1}

The following provides an error bound for approximation by $S_X(\Phi)$ where $X$ is quasi-uniform, and the error is measured as
\begin{equation*} 
 \mathcal{E}(f, S_X(\Phi))_p = \inf_{S \in S_X(\Phi)} \norm{f-S}_{L^p(\mathbb{R}^d)}
\end{equation*}
for $1 \leq p \leq \infty$.

\begin{theorem}\label{gre_per}
Let $X$ be a quasi-uniform set in $\mathbb{R}^d$.  Suppose $\Phi$ is an RBF satisfying Property 1' and Property 2, and let $f \in C_c^k(\mathbb{R}^d)$, then 
\begin{equation*} 
 \mathcal{E}(f, S_X(\Phi))_p \leq C h_X^\kappa \abs{f}_{W(L^p(\mathbb{R}^d),T)}
\end{equation*}
for any $1\leq p \leq \infty$.
\end{theorem}
\begin{proof}
 First, we define a sequence of approximants to $f$ by 
\begin{eqnarray*}
 F_n &:=& \int_{\mathbb{R}^d} Tf*\mu_n(t)K(\cdot,t)dt \\
&=& \sum_{\xi \in X}\Phi(\cdot-\xi) \int_{\mathbb{R}^d} Tf*\mu_n(t)A(t,\xi)dt.
\end{eqnarray*}
Note that this sum is finite due to the compact support of $Tf*\mu_n$ and the conditions imposed on $A(\cdot,\cdot)$. Now 
\begin{eqnarray*}
 \abs{f-F_n} &=& \abs{\int_{\mathbb{R}^d} Tf(t)G(\cdot-t)dt - \int_{\mathbb{R}^d} Tf*\mu_n(t)K(\cdot,t)dt} \\
&=& \abs{\int_{\mathbb{R}^d} Tf*\mu_n(t)\left(\Phi(\cdot-t) - K(\cdot,t)\right)dt + Tf*\nu_n}\\
&\leq& \int_{\mathbb{R}^d} \abs{Tf*\mu_n(t)} \abs{\Phi(\cdot-t) - K(\cdot,t)}dt + \abs{Tf*\nu_n},
\end{eqnarray*}
and by Property 2,
\begin{equation*}
 \abs{f-F_n} \leq Ch_X^{\kappa-d} \int_{\mathbb{R}^d}\abs{Tf*\mu_n(t)} \left(1+\frac{\norm{\cdot-t}_2}{h_X}\right)^{-l}dt + \abs{Tf*\nu_n}.
\end{equation*}
Therefore
\begin{eqnarray*}
 \norm{f-F_n}_p &\leq& Ch_X^{\kappa-d} \norm{\abs{Tf*\mu_n}* \left(1+\frac{\norm{\cdot}_2}{h_X}\right)^{-l}}_p +\norm{Tf*\nu_n}_p \\
&\leq& Ch_X^{\kappa} \norm{Tf}_p \norm{\mu_n} + \norm{Tf}_p \norm{\nu_n}
\end{eqnarray*}
 by the generalization of Young's inequality stated in equation \eqref{eq:young}. Hence, for $N$ sufficiently large (depending on $X$), we will have 
\begin{equation*} 
 \norm{f-F_N}_p \leq Ch_X^\kappa \norm{Tf}_p.
\end{equation*}
\qed
\end{proof}

Now that we have established this result for $C_c^k$, we would like to extend it to all of $W(L^p(\mathbb{R}^d),T)$.  This is accomplished in the following corollary using a density argument.

\begin{corollary}\label{gre_per_cor}
 Suppose $f \in W(L^p(\mathbb{R}^d),T)$, then 
\begin{equation*} 
 \mathcal{E}(f, S_X(\Phi))_p \leq C h_X^\kappa \abs{f}_{W(L^p(\mathbb{R}^d),T)}
\end{equation*}
\end{corollary}
\begin{proof}
 Let $f_n \in C_c^\infty$ be a sequence converging to $f$ in $W(L^p(\mathbb{R}^d),T)$, and let $F_n \in S_X(\Phi)$ be a sequence of approximants satisfying
\begin{equation*}
 \norm{f_n-F_n}_p \leq C h_X^\kappa \abs{f_n}_{W(L^p(\mathbb{R}^d),T)}.
\end{equation*}
Then
\begin{eqnarray*} 
 \norm{f-F_n}_p &\leq& \norm{f-f_n}_p + \norm{f_n-F_n}_p \\
&\leq& \norm{f-f_n}_p + C h_X^\kappa \abs{f_n}_{W(L^p(\mathbb{R}^d),T)}.
\end{eqnarray*}
Now since $f_n$ converges to $f$ in $W(L^p(\mathbb{R}^d),T)$, the result follows.
\qed
\end{proof}

\section{Special Cases}

In this section we will provide some examples, showing how to verify the properties listed above. However, we will first need to take a look at local polynomial reproductions.  As in \cite{r:dev:1}, this will be a crucial part of verifying Property 2.  To that end let $\Phi \in C^n(\mathbb{R}^d)$ be the RBF under consideration, and let $P$ denote the space of polynomials of degree at most $n-1$ on $\mathbb{R}^d$.  For a finite set $Y\subset\mathbb{R}^d$, let $\Lambda_Y$ be the set of extensions to $C(\mathbb{R}^d)$ of linear combinations of the point evaluation functionals $\delta_y:P|_Y\rightarrow \mathbb{R}$.  We now assume the following, cf. \cite[p.9]{r:dev:1}:
\begin{itemize}
 \item[(i)] There is a constant $C_1$ such that  $X(t) \subset B(t,C_1h_X)\cap X$ for all $t\in\mathbb{R}^d$.
 \item[(ii)] For all $t$, there exists $\lambda_t\in\Lambda_{X(t)}$ such that $\lambda_t$ agrees with $\delta_t$ on $P$.
 \item[(iii)] $\norm{\lambda_t}\leq C_2$ for some constant $C_2$ independent of $t$.
\end{itemize}
Based on these assumptions, $\lambda_t$ takes the form $\sum_{\xi \in X(t)} A(\xi,t)\delta_\xi$, and we can define a kernel approximant to $\Phi$ by
\begin{eqnarray*}
 K(x,t) &=& \lambda_t(\Phi(x-\cdot)) \\
   &=& \sum_{\xi \in X(t)} A(t,\xi) \Phi(x-\xi).
\end{eqnarray*}

One way to verify the validity of our assumptions is the following.
For each $x\in h_X\mathbb{Z}^d$, we denote by $Q_x$ the cube of side length $h_X$ centered at $x$.  This provides a partition of $\mathbb{R}^d$ into cubes.  We then associate to each cube $Q_x$ the ball $B(x,C_3h_X)$ for some constant  $C_3$, and for each $t$ in a fixed $Q_{x_0}$,  we define $X(t)= X\cap B(x_0,C_3h_X)$.  By choosing $C_3$ appropriately and bounding the possible values of $h_X$, one can show that there exists $\lambda_t$ satisfying the above properties, with $C_2=2$, for all $t \in Q_{x_0}$, cf. \cite[Chapter 3]{r:wen:2}.  As $x_0$ was arbitrary, the result holds for all $t\in\mathbb{R}^d$.  Note that by using the same set $X(t)$ for all $t$ in a given cube $Q_{x_0}$, we are able to choose the coefficients $A(\xi,t)$ so that they are continuous with respect to $t$ in the interior of the cube.  To see this, let $\{p_i\}_{i=1}^m$ be a basis for $P$ and let $B(x_0,C_3h_X)\cap X = \{\xi_j\}_{j=1}^N$.  We now define the matrix $M_{i,j}:=p_i(\xi_j)$ and the $t$ dependent vector $\beta(t)_i:=p_i(t)$.  Since $\lambda_t$ reproduces polynomials, there exists a vector $\alpha(t)$ such that $M\alpha(t)=\beta(t)$ for all $t\in Q_{x_0}$.  We could therefore choose a specific $\alpha(t)$ by means of a pseudo-inverse, i.e. $\alpha(t):=M^T(MM^T)^{-1}\beta(t)$.  In this form it is clear that $\alpha$ is a continuous function of $t$ and $\lambda_t:=\sum_{i=1}^N \alpha_i(t)\delta_{\xi_i}$ satisfies properties (i) and (ii).  Property (iii) follows from the fact that the pseudo-inverse gives the minimum norm solution.

\subsection{Sobolev Splines}

We are now in a position to prove Property 2 for the Sobolev Splines $G_\gamma$ where $\gamma$ is a positive even integer. In particular, we will show that there is a kernel $K_\gamma$ so that
\begin{equation*}
 \abs{G_\gamma(x-t)-K_\gamma(x,t)} \leq C h_X^{\gamma-d}\left(1+\frac{\norm{x-t}_2}{h_X}\right)^{-d-1}.
\end{equation*}
when $d$ is odd and obtain a similar result for $d$ even.  From the argument above, we know that there exists $\lambda_t$ that reproduces polynomials of degree $\gamma$ and satisfies the required form defined in the previous section.  Using $\lambda_t$, we define the kernel approximant to $G_\gamma$ by 
\begin{equation*}
 K_\gamma(x,t):=\lambda_t(G_\gamma(x-\cdot)).
\end{equation*}

Property 2 will only be verified for $t=0$ as all other cases work similarly.  First suppose that $\norm{x}_2\geq 2C_1h_X$, then $G_\gamma$ is analytic in $B(x, C_1h_X)$.  If we let $R$ be the degree $\gamma$ Taylor polynomial of $\Phi$ at $x$, then we get
\begin{equation*}
 \abs{G_\gamma(x)-K_\gamma(x,0)} =\abs{G_\gamma(x)-R(x) - \lambda_0(G_\gamma(x-\cdot)-R(x-\cdot))}, 
\end{equation*}
 and therefore 
\begin{equation*}
 \abs{G_\gamma(x)-K_\gamma(x,0)} \leq (1+C_2)\norm{G_\gamma-R}_{L^\infty(B(x,C_1h_X))}. 
\end{equation*}
From Taylor's remainder theorem, we get
\begin{equation*}
 \abs{G_\gamma(x)-K_\gamma(x,0)} \leq \frac{C}{(\gamma+1)!}h_X^{\gamma+1}\norm{G_\gamma}_{W^{\gamma+1,\infty}(B(x,C_1h_X))},
\end{equation*}
and applying proposition \ref{pr:ss2} gives
\begin{eqnarray}
 \abs{G_\gamma(x)-K_\gamma(x,0)} &\leq& C h_X^{\gamma+1} \abs{x}_2^{-d-1} \\
          &\leq& C h_X^{\gamma-d} \left(1+\frac{\norm{x}_2}{h_X}\right)^{-d-1}.\label{eq:large}
\end{eqnarray}

Now if $\norm{x}_2<2C_1h_X$ and $d$ is odd, we define $R$ to be the degree $\gamma-d-1$ Taylor polynomial of $G_\gamma$ at 0; if $d$ is even, let $R$ to be the degree $\gamma-d-2$ Taylor polynomial of $G_\gamma$ at 0.  For $d$ odd, we get
\begin{equation*}
 \abs{G_\gamma(x)-K_\gamma(x,0)} \leq C h_X^{\gamma-d}\norm{G_\gamma}_{W^{\gamma-d,\infty}(B(0,3C_1h_X))},
\end{equation*}
and by Proposition \ref{pr:ss2}, we have
\begin{equation*}
 \abs{G_\gamma(x)-K_\gamma(x,0)} \leq C h_X^{\gamma-d}.
\end{equation*}
By making use of the assumption $\norm{x}_2<2C_1h_X$, we finally get
 \begin{equation}\label{eq:small:odd}
 \abs{G_\gamma(x)-K_\gamma(x,0)} \leq C h_X^{\gamma-d}\left(1+\frac{\norm{x}_2}{h_X}\right)^{-d-1}.
\end{equation}
Similarly for $d$ even, we have
\begin{equation}\label{eq:small:even}
\abs{G_\gamma(x)-K_\gamma(x,0)} \leq C h_X^{\gamma-d-1}\left(1+\frac{\norm{x}_2}{h_X}\right)^{-d-1}.
\end{equation}

Putting these estimates together, equations (\ref{eq:large}), (\ref{eq:small:odd}), and (\ref{eq:small:even}) imply Property 2 for the Sobolev splines.  Note that Property 1' is clearly satisfied for $G=G_\gamma$ and $T=(1-\Delta)^{\gamma/2}$. 

\begin{theorem}
For any positive even integer $\gamma \geq d+1$ and any $1 \leq p < \infty$, the following rates are valid for approximation by the Sobolev splines $G_\gamma$ 
\begin{itemize}
 \item [i)] If $d$ is odd
\begin{equation*} 
  \mathcal{E}(f,S_X(G_\gamma))_p \leq C h_X^{\gamma} \abs{f}_{L^{\gamma,p}(\mathbb{R}^d)}
\end{equation*}
 \item [ii)] If $d$ is even
\begin{equation*} 
  \mathcal{E}(f,S_X(G_\gamma))_p \leq C h_X^{\gamma-1} \abs{f}_{L^{\gamma,p}(\mathbb{R}^d)}.
\end{equation*}
\end{itemize}
\end{theorem}

\subsection{Wendland Functions}

We now come to the reason for studying perturbations:  to derive approximation results for the Wendland functions.  We will begin by using the polynomial reproducing functionals $\lambda_t$ to verify Property 2.  Note that $\Phi_{d,k}(x-t)$ and $\lambda_t(\Phi_{d,k}(x-\cdot))$ are both zero for $\norm{x-t}_2>1+C_1h_X$.  Therefore in order to verify Property 2, it suffices to show that for $\norm{x-t}_2<1+C_1h_X$ we have $\abs{\Phi_{d,k}(x-t)-\lambda_0(\Phi_{d,k}(x-\cdot))} \leq Ch_X^{2k}$.  Similar to the Sobolev splines, we will only verify this inequality for $t=0$.  To begin, fix $x$, and let $R$ be the $2k-1$ degree Taylor polynomial of $\Phi_{d,k}$ at $x$. Then

\begin{eqnarray*}
\abs{\Phi_{d,k}(x)-\lambda_0(\Phi_{d,k}(x-\cdot))}&=& \abs{(\Phi_{d,k}-R)(x)-\lambda_0((\Phi_{d,k}-R)(x-\cdot))}\\
&\leq& \norm{\lambda_0}\norm{\Phi_{d,k}-R}_{L^\infty(B(x,C_1h_X))} \\
&\leq& C_{\Phi_{d,k}}h_X^{2k} \norm{\lambda_0} \norm{\Phi_{d,k}}_{W^{2k,\infty}(B(x,C_1h_X))}
\end{eqnarray*}
We have therefore shown that each $\Phi_{d,k}$ satisfies Property 2 with $\kappa=2k$.

The challenge in proving rates for the Wendland functions appears when trying to prove Property 1'.  To accomplish this, we can work in the Fourier domain, where the convolution becomes a standard product.  For example, to show that $G=\Phi*\mu$ for some $\mu \in M(\mathbb{R}^d)$, we can verify that $\hat{G}/\hat{\Phi}$ is the Fourier transform of some $\mu \in M(\mathbb{R}^d)$.  The difficulty lies in characterizing the space of Fourier transforms of $M(\mathbb{R}^d)$.  This is known to be a very difficult problem, cf. \cite{r:ben}.  However, in certain situations, we are able to make this determination. Our approach will be to work under the assumption that the functions being approximated have added smoothness, and we will additionally work in odd space dimension $d$, since we know the exact form of the Fourier transform of $\Phi_{d,k}$.  In this situation, we can choose our Green's function $G$ to be the Sobolev spline $G_{2d+2k+2}$. 
\begin{lemma}
 Let $G$ and $\Phi_{d,k}$ be defined as above with $k\in\mathbb{Z}^+$. In this case we have
\begin{equation*}
 \mu:=(\hat{G}/\hat{\Phi})^\vee \in L^1(\mathbb{R}^d)
\end{equation*}
\end{lemma}
\begin{proof}
 First, it is clear that $\hat{G}/\hat{\Phi}$ is a radial $L^1(\mathbb{R}^d)$ function, so $\mu \in L^\infty(\mathbb{R}^d)$.  Additionally, we can write
\begin{equation*}
 \mu(x) = r^{-(d-2)/2}\int_0^\infty \frac{(1+t^2)^{-(d+1)/2}}{(1+t^2)^{(d+2k+1)/2}\hat{\Phi}(t)}t^{d/2}J_{(d-2)/2}(rt)dt
\end{equation*}
where $J$ is a Bessel function of the first kind.  In this form, we can bound $\mu$ for $r$ large, and hence show that it has sufficient decay to be in $L^1(\mathbb{R}^d)$.  Verifying the decay of $\mu$ is accomplished using the fact that 
\begin{equation*}
 \frac{d}{dz}\{z^\nu J_\nu(z)\}=z^\nu J_{\nu-1}(z)
\end{equation*}
and integrating by parts $(d+5)/2$ times.  Note that Proposition \ref{pr:wend_ft_1} is used to bound the decay at each step.
\qed
\end{proof}

With $\mu$ defined as in the previous lemma, we get
\begin{equation*}
 G_{2d+2k+2}=\Phi_{d,k}*\mu,
\end{equation*}
and by choosing $\mu_n$ to be the restriction of $\mu$ to $B(0,n)$, we can see that Property 1' is satisfied. Therefore we have the following theorem.

\begin{theorem}\label{thm:wend:rd}
 For any positive integer $k$ and any odd space dimension $d$, we have
\begin{equation*}
 \mathcal{E}(f,S_X(\Phi_{d,k}))_p \leq C h_X^{2k} \abs{f}_{L^{2d+2k+2,p}}
\end{equation*}
for $1 \leq p < \infty$.
\end{theorem}

Upon inspection, one can can see that we have to assume a higher order of smoothness on the functions being approximated than the rate of approximation.  Notice that the approximated functions $f$ are assumed to be in $L^{2d+2k+2,p}$, but the rate of approximation is only $h_X^{2k}$. Despite this shortcoming, the derived rate is better than what is currently known in certain cases.  In \cite{r:wen:2}, Wendland determines a rate of $h_X^{(d+2k+1)/2-d(1/2-1/p)_+}$ for bounding the $L^p(\Omega)$ error in approximating functions from $W^{(d+2k+1)/2,2}(\Omega)$, where $\Omega$ is a bounded domain in $\mathbb{R}^d$.  First, notice that when approximating arbitrarily smooth functions and $1\leq p \leq 2$, Theorem \ref{thm:wend:rd} provides a better rate as long as $k>(d+1)/2$.  However, the improvement is most noticeable when $p$ is close to $\infty$.  Also note that the prior rate suffers the same drawback as Theorem \ref{thm:wend:rd}; when $p=1$, the $L^1$ smoothness needed to be in $W^{(d+2k+1)/2,2}(\Omega)$ is greater than $d+2k+1$.

\subsection{Improved Estimates for the Wendland Functions}

In the generality of Theorem \ref{thm:wend:rd} the assumed smoothness is greater than the rate of approximation.  One source of this disparity was that when verifying Property 1', we compared $\Phi_{d,k}$ to the Green's function $G_{2d+2k+1}$ which has a Fourier transform with a faster rate of decay.  In this section, we propose an alternate method and show that for $d=1$ we have $\Phi_{1,k}=G_{2k+2}*\mu$ for some finite Borel measure $\mu$ and hence improve the error estimates.
Notice that the Fourier transforms of $\Phi_{1,k}$ and $G_{2k+2}$ have the same rate of decay, so in this situation, we are no longer losing orders of approximation from the perturbation technique.

  Our approach to verifying Property 1' will be to view $M(\mathbb{R})$ as a Banach algebra and use an invertibility result of Benedetto.
The first step is to show that $\hat{\Phi}/\hat{G}$ is the Fourier transform of some $\mu \in M(\mathbb{R})$.  Then the following theorem will give conditions for $\hat{\mu}^{-1}=\hat{G}/\hat{\Phi}$ being the Fourier transform of an element of $M(\mathbb{R})$.  Recall the decomposition of measures:  a Borel measure $\mu$ can be written as $\mu_a+\mu_s+\mu_d$.  Using this notation, we state the following theorem of Benedetto. Note that the author proved this theorem in a more general setting with $\mathbb{R}$ replaced by an arbitrary locally compact Abelian group.

\begin{theorem}\label{benn}\textbf{\em (\cite[Theorem 2.4.4]{r:ben})}
  Let $\mu \in M(\mathbb{R})$ such that $\abs{\hat{\mu}}$ never vanishes and
\begin{equation*}
 \norm{\mu_s} < \inf_{x \in \mathbb{R}}\abs{ \hat{\mu}_d(x)}. 
\end{equation*}
Then $\hat{\mu}^{-1}$ is the Fourier transform of an element of $M(\mathbb{R})$.
\end{theorem}

We now fix $k\geq 1$, and we must select a Green's function $G$ and corresponding differential operator $T$ so that $\hat{\Phi}_{1,k}/\hat{G}$ satisfies the necessary conditions. Considering the decay of $\hat{\Phi}_{1,k}$, we choose $\hat{G}=(1+\abs{\cdot}^{2k+2})^{-1}$ and $T=(1+(-1)^{k+1}\Delta^{k+1}$). Note that for $1 \leq p <\infty$ we have $W(L^p(\mathbb{R}),T)$ is equivalent to the spaces $W^{2k+2,p}(\mathbb{R})$ and $L^{2k+2,p}(\mathbb{R})$.  

With this Green's function $G$, we can verify the hypotheses of Theorem \ref{benn}.  By Proposition \ref{pr:wend_ft}, we know that for some $B_k>0$ and $C \in \mathbb{R}$ we have

\begin{equation*}
 \frac{\hat{\Phi}_{1,k}}{\hat{G}}(x) = \hat{\Phi}_{1,k}+ B_k\left(\frac{1}{k!}+\frac{(-1)^{k+1}}{k!2^k}\cos(x)-C\frac{\sin(x)}{x}+\hat{h}(x)\right)
\end{equation*}
where $\hat{h}$ is the Fourier transform of an $L^1$ function, $h$.  The fact that $\hat{\Phi}_{1,k}/\hat{G}$ is positive follows from the positivity of $\hat{\Phi}_{1,k}$, cf. \cite[Chapter 10]{r:wen:2}.  Now let $\mu$ be the measure defined by 

\begin{equation*}
 \mu = \Phi_{1,k} + \sqrt{2\pi}B_k\left(\frac{1}{k!}\delta_0 + \frac{(-1)^{k+1}}{k!2^{k+1}}(\delta_{-1}+\delta_1) -\frac{C}{2}  \raisebox{\depth}{\(\chi\)}_{[-1,1]}+\frac{1}{\sqrt{2\pi}}h\right).
\end{equation*}
Then $\mu$ is a finite Borel measure with $\hat{\mu}=\hat{\Phi}_{1,k}/\hat{G}$.  Additionally, $\mu$ has no singular continuous part, and 
\begin{equation*}
 \abs{\hat{\mu}_d(\omega)}= B_k\left(\frac{1}{k!}+\frac{(-1)^{k+1}}{k!2^k}\cos(\omega)\right) \geq \frac{B_k}{2k!}
\end{equation*}
Therefore Theorem \ref{benn} implies that $1/\hat{\mu}= \hat{G}/\hat{\Phi}_{1,k}$ is the Fourier transform of some finite Borel measure $\tilde{\mu}$.  By letting $\tilde{\mu}_n$ be the restriction of $\tilde{\mu}$ to $B(0,n)$, we can see that 
\begin{eqnarray*}
 G &=& \Phi_{1,k}*\tilde{\mu} \\
   &=& \Phi_{1,k}*\tilde{\mu}_n + \Phi_{1,k}*(\tilde{\mu}-\tilde{\mu}_n),
\end{eqnarray*}
and hence corollary \ref{gre_per_cor} applies.
 
\begin{theorem}
 Let $\Phi_{1,k}$ be as above, and let $X$ be a quasi-uniform subset of $\mathbb{R}$. Then for $f \in{L^{2k+2,p}(\mathbb{R})}$ and $1 \leq p <\infty$, we have
\begin{equation*}
 \mathcal{E}(f, S_X(\Phi_{1,k}))_p \leq C h_X^{2k} \abs{f}_{L^{2k+2,p}(\mathbb{R})}
\end{equation*}
\end{theorem}

By examining this example, we can determine what the analogous result would be for the remaining Wendland functions.  If $k\geq 1$, then it is known that  
\begin{equation*}
 c_1(1+\norm{\omega}_2^2)^{-(d+2k+1)/2} \leq \hat{\Phi}_{d,k}(\omega) \leq c_2(1+\norm{\omega}_2^2)^{-(d+2k+1)/2}
\end{equation*}
for some positive constants $c_1$ and $c_2$, cf. \cite[Theorem 10.35]{r:wen:2}.  Therefore, we could choose $G$ to be the Sobolev spline of order $d+2k+1$, so that $\hat{G}$ and $\hat{\Phi}_{d,k}$ have similar decay.  In this situation, $\hat{G}/\hat{\Phi}_{d,k}$ is a continuous function that is bounded above and bounded away from 0.  If we could show that $\hat{G}/\hat{\Phi}_{d,k}$ is the Fourier transform of an element of $M(\mathbb{R}^d)$, then we could apply corollary \ref{gre_per_cor} to obtain the following.

\begin{conjecture}
 Let $\Phi_{d,k}$ be as above with $k\geq1$, and let $X$ be a quasi-uniform subset of $\mathbb{R}^d$. Then for $f \in{L^{d+2k+1,p}(\mathbb{R}^d)}$ and $1 \leq p <\infty$, we have
\begin{equation*}
 \mathcal{E}(f, S_X(\Phi_{d,k}))_p \leq C h_X^{2k} \abs{f}_{L^{d+2k+1,p}(\mathbb{R}^d)}.
\end{equation*}
\end{conjecture}

\end{document}